\documentclass[11pt,twoside]{article}
\usepackage{natbib}
\usepackage{upgreek}
\usepackage{amsthm,amssymb,amsmath}
\usepackage[small,compact]{titlesec}
\titlespacing{\section}{0pt}{*0}{*0}
\titlespacing{\subsection}{0pt}{*0}{*0}
\titlespacing{\subsubsection}{0pt}{*0}{*0}

\usepackage{bm}
\usepackage{mathrsfs}
\usepackage{ctable}
\usepackage{times}
\usepackage[left=1in,top=1.2in,right=1in,bottom=1.2in,nohead]{geometry}
\newcommand{\Var}{\mathrm{Var}}
\newcommand{\E}{\mathrm{E}}
\newcommand{\norm}[1]{\left\lVert#1\right\rVert}
\newtheorem{theorem}{Theorem}
\newtheorem{lemma}{Lemma}
\newtheorem{corollary}{Corollary}
\newtheorem{remark}{\textit{Remark}}
\date{}
\title{\textbf{Bayesian inference for higher order ordinary differential equation models}}
\author{Prithwish Bhaumik and Subhashis Ghosal\\North Carolina State University}
\begin{document}
\maketitle
\begin{abstract}
Often the regression function appearing in fields like economics, engineering, biomedical sciences obeys a system of higher order ordinary differential equations (ODEs). The equations are usually not analytically solvable. We are interested in inferring on the unknown parameters appearing in the equations. Significant amount of work has been done on parameter estimation in first order ODE models. \citet{bhaumik2014bayesiant} considered a two-step Bayesian approach by putting a finite random series prior on the regression function using B-spline basis. The posterior distribution of the parameter vector is induced from that of the regression function. Although this approach is computationally fast, the Bayes estimator is not asymptotically efficient. \citet{bhaumik2014efficient} remedied this by directly considering the distance between the function in the nonparametric model and a Runge-Kutta (RK$4$) approximate solution of the ODE while inducing the posterior distribution on the parameter. They also studied the direct Bayesian method obtained from the approximate likelihood obtained by the RK4 method. In this paper we extend these ideas for the higher order ODE model and establish Bernstein-von Mises theorems for the posterior distribution of the parameter vector for each method with $n^{-1/2}$ contraction rate.
\end{abstract}


\section{Introduction}
Consider a regression model $Y=f_{\bm\theta}(x)+\bm\varepsilon$ with unknown parameter $\bm\theta\in\bm\Theta\subseteq\mathbb{R}^p$ and $x\in[0,1]$. The functional form of $f_{\bm\theta}(\cdot)$ is not known but $f_{\bm\theta}(\cdot)$ satisfy a $q^{th}$ order ordinary differential equation (ODE) given by
\begin{eqnarray}
F\left(t, f_{\bm\theta}(t), \frac{df_{\bm\theta}(t)}{dt},\ldots, \frac{d^qf_{\bm\theta}(t)}{dt^q},\bm\theta\right)=0,\label{intro_hode}
\end{eqnarray}
where $F$ is a known sufficiently smooth real-valued function in its arguments which we shall refer to as the binding function. Higher order ODE models are encountered in different fields. For example, the system of ODE describing the concentrations of glucose and hormone in blood is
\begin{eqnarray}
\frac{dg(t)}{dt}&=&-m_1g(t)-m_2h(t)+J(t),\label{glucose}\\
\frac{dh(t)}{dt}&=&-m_3h(t)+m_4g(t),\label{hormone}
\end{eqnarray}
where $g(t)$ and $h(t)$ denote the glucose and hormone concentrations at time $t$ respectively. Here the function $J(t)$ is known and $m_1, m_2, m_3$ and $m_4$ are unknown parameters. If we have only measurements on $g(t)$, we can differentiate both sides of \eqref{glucose} to obtain the second order ODE
\begin{eqnarray}
\frac{d^2g(t)}{dt^2}+2\alpha\frac{dg(t)}{dt}+\omega^2_0g(t)=S(t),\nonumber
\end{eqnarray}
where $\alpha=(m_1+m_3)/2$, $\omega^2_0=m_1m_3+m_2m_4$ and $S(t)=m_3J(t)+dJ(t)/dt$. Another popular example is the Van der Pol oscillator used in physical and biological sciences. The oscillator obeys the second order ODE $$\frac{d^2f_{\theta}(t)}{dt^2}-\theta(1-f^2_{\theta}(t))\frac{df_{\theta}(t)}{dt}+f_{\theta}(t)=0.$$

A related problem is a stochastic differential equation model where a signal is continuously observed in time with a noise process typically driven by a Brownian motion. \citet{bergstrom1983gaussian,bergstrom1985estimation,bergstrom1986estimation} used the maximum likelihood estimation (MLE) technique to estimate the parameters involved in  higher order stochastic differential equation given by $$\frac{d^qy(t)}{dt^q}=A_1(\bm\theta)\frac{d^{q-1}y(t)}{dt^{q-1}}+\cdots+A_{q-1}(\bm\theta)\frac{dy(t)}{dt}+A_qy(t)+b(\bm\theta)+z(t)+W(t),$$ where $A_j(\cdot), j=1,\ldots,q$ and $b(\cdot)$ are functions on $\bm\Theta$ and $W(\cdot)$ is the noise process \citep[page 342]{aarlin1981second}. Here $z(\cdot)$ is a non-random function. \citet{bergstrom1983gaussian} showed that the maximum likelihood estimator of $\bm\theta$ is asymptotically normal and asymptotically efficient. An efficient algorithm was given in \citet{bergstrom1985estimation} to compute the Gaussian likelihood for estimating the parameters involved in a non-stationary higher order stochastic ODE. Appropriate linear transformations are used in this algorithm to avoid the computation of the covariance matrix of the observations.

In this paper we develop three Bayesian approaches for inference on $\bm\theta$. In our first approach we use Runge-Kutta method to obtain an approximate solution $f_{\bm\theta,r_n}(\cdot)$ using $r_n$ grid points, $n$ being the number of observations and then construct an approximate likelihood and obtain the posterior distribution of $\bm\theta$, using the prior of $\bm\theta$. In another approach we assign prior on the coefficient vector $\bm\beta$ of the B-spline approximation $f(\cdot,\bm\beta)$ of the regression function. We define $\bm\theta$ as $\arg\min_{\bm\eta\in\bm\Theta}\int_0^1\left(f(t,\bm\beta)-f_{\bm\eta,r_n}(t)\right)^2g(t)dt$ and induce a posterior distribution of $\bm\theta$ using the posterior distribution of $\bm\beta$. Here $g(\cdot)$ is an appropriate weight function. The third approach is a generalization of the two-step approach. We use the B-spline approximation of the regression function and define $$\bm\theta=\arg\min_{\bm\eta\in\bm\Theta}\int_0^1\left|F\left(\cdot, f(\cdot,\bm\beta), \frac{df(t,\bm\beta)}{dt},\ldots, \frac{d^qf(t,\bm\beta)}{dt^q},\bm\eta\right)\right|^2w(t)dt,$$ where the weight function and its first $(q-1)$ derivatives vanish at $0$ and $1$. For the sake of simplicity we have assumed the regression function to be one dimensional. Extension to the multidimensional case where the binding function $F$ is also vector-valued, can be carried out similarly.

The rest of the paper is organized as follows. The notations are described in Section \ref{notation}. Section \ref{prelim_hode} contains some preliminaries of Runge-Kutta method. The model assumptions and prior specifications are given in Section \ref{model_hode}. Section \ref{method_hode} contains the descriptions of the estimation methods used. The main results are given in Section \ref{result_hode}. In Section \ref{sim_hode} we have carried out a simulation study. Proofs of the results are given in Section \ref{proof_hode}.

\section{Notations and preliminaries}\label{notation}
We describe a set of notations to be used in this paper. Boldfaced letters are used to denote vectors and matrices. For a matrix $\bm A$, the symbols $\bm A_{i,}$ and $\bm A_{,j}$ stand for the $i^{th}$ row and $j^{th}$ column of $\bm A$ respectively. The identity matrix of order $p$ is denoted by $\bm I_p$. We use the symbols $\mathrm{maxeig}(\bm A)$ and $\mathrm{mineig}(\bm A)$ to denote the maximum and minimum eigenvalues of the matrix $\bm A$ respectively.
The $L_2$ norm of a vector $\bm x\in\mathbb{R}^p$ is given by $\|\bm x\|=\left(\sum_{i=1}^px_i^2\right)^{1/2}$. For a function $\bm\phi(\cdot)$ and a vector $\bm x\in\mathbb{R}^p$, we denote $D_{\bm x}\bm\phi=\frac{\partial}{\partial\bm x}\bm\phi=\left(\frac{\partial\bm\phi}{\partial x_1},\ldots,\frac{\partial\bm\phi}{\partial x_p}\right).$ The notation $f^{(r)}(\cdot)$ stands for the $r^{th}$ order derivative of a function $f(\cdot)$, that is, $f^{(r)}(t)=\frac{d^r}{dt^r}f(t)$. For the function $\bm\theta\mapsto f_{\bm\theta}(\cdot)$, the notation $\dot{f}_{\bm\theta}(\cdot)$ implies $\frac{\partial}{\partial\bm\theta}f_{\bm\theta}(\cdot)$. Similarly, we denote $\ddot{f}_{\bm\theta}(\cdot)=\frac{\partial^2}{\partial\bm\theta^2}f_{\bm\theta}(\cdot)$. A vector valued function is represented by the boldfaced symbol $\bm f(\cdot)$. Let us define $\|\bm f\|_g=(\int_0^1\|\bm f(t)\|^2g(t)dt)^{1/2}$ for $\bm f:[0,1]\mapsto\mathbb{R}^p$ and $g:[0,1]\mapsto[0,\infty)$. The weighted inner product $\int_0^1\bm f^T_1(t)\bm f_2(t)g(t)dt$ of two vector-valued functions $\bm f_1(\cdot)$ and $\bm f_2(\cdot)$ with the corresponding weight function $g(\cdot)$ is denoted by $\langle\bm f_1,\bm f_2\rangle_g$. For numerical sequences $a_n$ and $b_n$, $a_n=o(b_n)$, $b_n\gg a_n$ and $a_n\ll b_n$ all mean $a_n/b_n\rightarrow0$ as $n\rightarrow\infty$. The notations $a_n=O(b_n)$, $a_n\lesssim b_n$ are used to indicate that $a_n/b_n$ is bounded, and $a_n\asymp b_n$ refers to the occurrence of both $a_n=O(b_n)$ and $b_n=O(a_n)$. The symbol $o_P(1)$ stands for a sequence of random variables converging in $P$-probability to zero, whereas $O_P(1)$ stands for a stochastically bounded sequence of random variables. Given a sample $\{ X_i:\,i=1,\ldots,n\}$ and a measurable function $\psi(\cdot)$, we define $\mathbb{P}_n\psi=n^{-1}\sum_{i=1}^n\psi(X_i)$. The symbols $\E(\cdot)$ and $\Var(\cdot)$ stand for the mean and variance respectively of a random variable. We use the notation $\mathbb{G}_n\psi$ to denote $\sqrt{n}\left(\mathbb{P}_n\psi-\E\psi\right)$. The total variation distance between the probability measures $P$ and $Q$ defined on $\mathbb{R}^p$ is given by $\|P-Q\|_{TV}=\sup_{B\in\mathscr{R}^p}|P(B)-Q(B)|$, $\mathscr{R}^p$ being the Borel $\sigma$-field on $\mathbb R^p$. Given a set $E$, the symbol $C^m(E)$ stands for the class of functions defined on $E$ having first $m$ continuous partial derivatives with respect to its arguments on some open set containing $E$. For a set $A$, the notation $\mathrm{l}\!\!\!1\{A\}$ stands for the indicator function for belonging to $A$. The symbol $:=$ means equality by definition.

\section{Preliminaries of Runge-Kutta method for higher order ODE}\label{prelim_hode}
Often the differential equation has the form
\begin{eqnarray}
F\left(t, f_{\bm\theta}(t), f^{(1)}_{\bm\theta}(t),\ldots, f^{(q)}_{\bm\theta}(t),\bm\theta\right)=f^{(q)}_{\bm\theta}(t)-H\left(t, f_{\bm\theta}(t), f^{(1)}_{\bm\theta}(t),\ldots, f^{(q-1)}_{\bm\theta}(t),\bm\theta\right)=0\nonumber
\end{eqnarray}
with initial conditions $f^{(\nu)}_{\bm\theta}(0)=c_{\nu}$ for $\nu=0,\ldots,q-1$, $H$ being known. Note that $t$ can be treated as a state variable $\chi(t)=t$ which satisfies the $q^{th}$ order ODE $\chi^{(q)}(t)=0$ with initial conditions $\chi(0)=0,\,\,\chi^{(1)}(0)=1$ and $\chi^{(j)}(0)=0$ for $j=2,\ldots,q-1$. Denoting $\bm\psi_{\bm\theta}(\cdot)=(f_{\bm\theta}(\cdot),\chi(\cdot))$, we can rewrite the ODE as $$\bm\psi^{(q)}_{\bm\theta}(t)=\bm H\left(\bm\psi_{\bm\theta}(t),\ldots,\bm\psi^{(q-1)}_{\bm\theta}(t)\right),$$ where $\bm H=(H(\cdot),0)$. Given $r_n$ equispaced grid points $a_1=0,a_2,\ldots,a_{r_n}$ with common difference $r^{-1}_n$, the approximate solution to \eqref{intro_hode} is given by $\bm\psi_{\theta,r_n}(\cdot)=(f_{\bm\theta,r_n}(\cdot),\chi_{r_n}(\cdot))$, where $r_n$ is chosen so that $r_n\gg\sqrt{n}$. Here $n$ denotes the number of observations. Let $\bm z_k=(\bm\psi_{\bm\theta,r_n}(a_k),  \bm\psi^{(1)}_{\bm\theta,r_n}(a_k),\ldots, \bm\psi^{(q-1)}_{\bm\theta,r_n}(a_k))$ stand for the vector formed by the function $\psi_{\bm\theta,r_n}$ and its $(q-1)$ derivatives at the $k^{th}$ grid point $a_k$ for $k=1,\ldots,r_n$. For $\nu=1,2,\ldots,q-1$ we define
\begin{eqnarray}
\bm T^\nu(a_k,\bm z_k,r_n)&:=&\bm\psi_{\bm\theta,r_n}^{(\nu)}(a_k)+\frac{1}{2!r_n} \bm\psi_{\bm\theta,r_n}^{(\nu+1)}(a_k)+\cdots+\frac{1}{r_n^{(q-\nu-1)}(q-\nu)!} \bm\psi_{\bm\theta,r_n}^{(q-1)}(a_k),\nonumber\\
\bm T^q(a_k,\bm z_k,r_n)&:=&\bm 0.\nonumber
\end{eqnarray}
Now let $\bm k_{\rho}(a_k):=\bm H(\bm U^1,\bm U^2,\ldots,\bm U^q)$ with $\bm U^1,\bm U^2,\ldots,\bm U^q$ being given in Table \ref{arg_table}.
$$\rm{``Table\,\ref{arg_table}\,here"}$$
Following equation $(4.16)$ of \citet[page 169]{henrici1962discrete} we define
\begin{eqnarray}
\bm\Phi^\nu(a_k,\bm z_k,r_n):=\bm T^\nu(a_k,\bm z_k,r_n)+\frac{1}{r_n^{(q-\nu)}(q-\nu+1)!}\sum_{\rho=1}^4\gamma_{\nu\rho}\bm k_{\rho}(a_k),\nonumber
\end{eqnarray}
where the coefficients $\gamma_{\nu\rho}$ are given by
\begin{eqnarray}
\gamma_{\nu1}=\frac{(q-\nu+1)^2}{(q-\nu+2)(q-\nu+3)},\nonumber\\
\gamma_{\nu2}=\gamma_{\nu3}=\frac{2(q-\nu+1)}{(q-\nu+2)(q-\nu+3)},\nonumber\\
\gamma_{\nu4}=\frac{1-q+\nu}{(q-\nu+2)(q-\nu+3)}\nonumber
\end{eqnarray}
for $\nu=1,\ldots,q$. Then the sequence $z_k$, $k=1,\ldots,r_n$ can be constructed by the recurrence relation $$\bm z_{k+1}=\bm z_{k}+r^{-1}_n\left(\bm\Phi^1(a_k,\bm z_k,r_n),\ldots,\bm\Phi^q(a_k,\bm z_k,r_n)\right)^T.$$
By the proof of Theorem 4.2 of \citet[page 174]{henrici1962discrete}, we have
\begin{eqnarray}
\sup_{t\in[0,1]}\left|f_{\bm\theta}(t)-f_{\bm\theta,r_n}(t)\right|=O(r^{-1}_n),\,
\sup_{t\in[0,1]}\left|\frac{\partial}{\partial\bm\theta}f_{\bm\theta}(t)-\frac{\partial}{\partial\bm\theta}f_{\bm\theta,r_n}(t)\right|=O(r^{-1}_n).\label{rk4}
\end{eqnarray}

\section{Model description and prior specification}\label{model_hode}
Now we formally describe the model. The proposed model is given by
\begin{eqnarray}
Y_{i}=f_{\bm\theta}(X_i)+\varepsilon_{i},\,i=1,\ldots,n,\label{prop_hode}
\end{eqnarray}
where $\bm\theta\subseteq\bm\Theta$, which is a compact subset of $\mathbb{R}^p$. The function $f_{\bm\theta}(\cdot)$ is $q$ times differentiable on an open set containing $[0,1]$ and satisfies the system of ODE given by
\begin{eqnarray}
F\left(t,f_{\bm\theta}(t),f^{(1)}_{\bm\theta}(t),\ldots,f^{(q)}_{\bm\theta}(t),\bm\theta\right)=0.\label{diff_hode}
\end{eqnarray}
Let for a fixed $\bm\theta$, $F(\cdot,\cdot,\bm\theta)\in C^{m-q+1}((0,1)\times\mathbb{R}^{q+1})$ for some integer $m\geq q$. Then, by successive differentiation we have $f_{\bm\theta}\in C^m((0,1))$. We also assume that the function $\bm\theta\mapsto f_{\bm\theta}(\cdot)$ is two times continuously differentiable. The true regression function is given by $f_0(\cdot)$ which does not necessarily lie in $\{f_{\bm\theta}:\bm\theta\in\bm\Theta\}$. Moreover we assume that $f_0\in C^m([0,1])$.
Let $\varepsilon_{i}$ be independently and identically distributed with mean zero and finite moment generating function for $i=1,\ldots,n$. Let the common variance be $\sigma^2_0$. We use $N(0,\sigma^2)$ as the working model for the error, which may be different from the true distribution of the errors. We treat $\sigma^2$ as an unknown parameter and assign an inverse gamma prior on $\sigma^2$ with shape and scale parameters $a$ and $b$ respectively. Additionally it is assumed that $X_i\stackrel{iid}\sim G$ with density $g$.\\
\indent Let us denote $\bm Y=(Y_1,\ldots,Y_n)^T$ and $\bm X=(X_1,\ldots,X_n)^T$. The true joint distribution of $(X_i,\varepsilon_i)$ is denoted by $P_0$.

\section{Methodology}\label{method_hode}
Now we describe the three different approaches of inference on $\bm\theta$ used in this paper.
\subsection{Runge-Kutta Sieve Bayesian Method (RKSB)}
For RKSB we denote $\bm\gamma=(\bm\theta,\sigma^2)$. The approximate likelihood of the sample $\{(X_i,Y_{i,}):\,i=1,\ldots,n\}$ is given by $L^*_n({\bm\gamma})=\prod_{i=1}^np_{\bm\gamma,n}(X_i,Y_{i,})$,
where
\begin{eqnarray}
p_{\bm\gamma,n}(X_i, Y_i)=(\sqrt{2\pi}\sigma)^{-1}\exp\{-(2\sigma^2)^{-1}| Y_{i}- f_{\bm\theta,r_n}(X_i)|^2\}g(X_i).\label{approx}
\end{eqnarray}
We also denote
\begin{eqnarray}
p_{\bm\gamma}(X_i, Y_i)=(\sqrt{2\pi}\sigma)^{-1}\exp\{-(2\sigma^2)^{-1}|Y_{i}-f_{\bm\theta}(X_i)|^2\}g(X_i).\label{actual}
\end{eqnarray}
The true parameter $\bm\gamma_0:=(\bm\theta_0,\sigma^2_*)$ is defined as $\bm\gamma_0=\arg\max_{\bm\gamma\in\bm\Theta\times(0,\infty)}P_0\log p_{\bm\gamma}$, which takes into account the natural requirement that if errors are normally distributed and $f_{\bm\theta_0}(\cdot)$ is the true regression function, then $\bm\gamma_0=(\bm\theta_0,\sigma^2_0)$, where $\sigma^2_0$ is the true value of the error variance. We denote by $\ell_{\bm\gamma}$ and $\ell_{\bm\gamma,n}$ the log-likelihoods with respect to \eqref{actual} and \eqref{approx} respectively. If $\bm\gamma_0$ is the unique maximizer of $P_0\log p_{\bm\gamma}$, we get
\begin{eqnarray}
\int_0^1\dot{f}^T_{\bm\theta_0}(t)\left(f_0(t)-f_{\bm\theta_0}(t)\right)g(t)dt=\bm 0,\,
\sigma^2_*=\sigma^2_0+\int_0^1|f_0(t)-f_{\bm\theta_0}(t)|^2g(t)dt.\label{deriv_10_hode}
\end{eqnarray}
We assume that the sub-matrix of the Hessian matrix of $-P_0\log p_{\bm\gamma}$ at $\bm\gamma=\bm\gamma_0$ given by
\begin{eqnarray}
\int_0^1\left(\left.\dot{f}^T_{\bm\theta_0}(t)\dot{f}_{\bm\theta_0}(t)-\frac{\partial}{\partial\bm\theta}\left(\dot{ f}^T_{\bm\theta}(t)\left(f_0(t)-f_{\bm\theta_0}(t)\right)\right)\right|_{\bm\theta=\bm\theta_0}\right)g(t)dt
\end{eqnarray}
is positive definite. Note that this condition is automatically satisfied when $f_{\bm\theta_0}(\cdot)$ is the true regression function. The prior measure on $\bm\Theta$ is assumed to have a Lebesgue-density continuous and positive on a neighborhood of $\bm\theta_0$. The prior distribution of $\bm\theta$ is assumed to be independent of that of $\sigma^2$. The joint prior measure is denoted by $\Pi$ with corresponding density $\pi$. We obtain the posterior of $\bm\gamma$ using the approximate likelihood given by \eqref{approx}.
\subsection{Runge-Kutta Two-step Bayesian Method (RKTB)}
In the RKTB approach, the proposed model is embedded in the nonparametric regression model
\begin{eqnarray}
\bm Y=\bm X_n\bm\beta+\bm\varepsilon,\label{np_hode}
\end{eqnarray}
where $\bm X_n=(\!(N_{j}(X_i))\!)_{1\leq i\leq n,1\leq j\leq k_n+m-1}$, $\{N_{j}(\cdot)\}_{j=1}^{k_n+m-1}$ being the B-spline basis functions of order $m$ with $k_n-1$ interior knots, see Chapter IX of \citet{de1978practical}. We assume for a given $\sigma^2$
\begin{equation}
\bm\beta\sim N_{k_n+m-1}(\bm 0,\sigma^2n^{2}k^{-1}_n\bm I_{k_n+m-1}).\label{prior_hode}
\end{equation}
A simple calculation yields that the conditional posterior distribution of $\bm\beta$ given $\sigma^2$ is
\begin{eqnarray}
N_{k_n+m-1}\left({\left({\bm X^T_n\bm X_n}+n^{-2}k_n\bm I_{k_n+m-1}\right)}^{-1}{\bm X^T_n\bm Y},\sigma^2{\left({\bm X^T_n\bm X_n}+n^{-2}k_n\bm I_{k_n+m-1}\right)}^{-1}\right).\label{posterior_hode}
\end{eqnarray}
For a given parameter $\bm\eta\in\bm\Theta$ let
\begin{eqnarray}
R_{f,n}(\bm\eta)=\left\{\int_0^1|f(t,\bm\beta)-f_{\eta,r_n}(t)|^2g(t)dt\right\}^{1/2},\,
R_{f_0}(\bm\eta)=\left\{\int_0^1|f_0(t)-f_{\eta}(t)|^2g(t)dt\right\}^{1/2},\nonumber
\end{eqnarray}
where $f(t,\bm\beta)=\bm\beta^T\bm N(t)$ with $\bm N(\cdot)=(N_{1}(\cdot),\ldots,N_{k_n+m-1}(\cdot))^T$. Now we define the parameter $\bm\theta$ by $\bm\theta=\arg\min_{\bm\eta\in\bm\Theta}R_{ f,n}(\bm\eta)$ as any minimizer and induce posterior distribution on $\bm\Theta$ through the posterior of $\bm\beta$ given by \eqref{posterior_hode}. Thus we extend the definition of $\bm\theta$ beyond the differential equation model. Let us denote $\bm\theta_0=\arg\min_{\bm\eta\in\bm\Theta}R_{f_0}(\bm\eta)$. Note that in well-specified case when $f_{\bm\theta_0}(\cdot)$ is the true regression function with corresponding true parameter $\bm\theta_0$, then the minima is automatically located at $\bm\theta_0$. We assume that
\begin{equation}
\text{for all}\,\epsilon>0,\,\,\inf_{\bm\eta:\left\|\bm\eta-\bm\theta_0\right\|\geq\epsilon}R_{f_0}(\bm\eta)>R_{f_0}(\bm\theta_0),\label{assmp_rktb_hode}
\end{equation}
that is, $R_{f_0}(\cdot)$ has a well separated unique minima at some point $\bm\theta_0$.
\subsection{Two-step Bayesian Method}
Here we use the same nonparametric model and prior specifications as in RKTB, but define the true parameter $\bm\theta_0$ as the unique minimizer of $\bm\eta\mapsto\|F(\cdot,f_0(\cdot),f^{(1)}_0(\cdot),\ldots,f^{(q)}_0(\cdot),\bm\eta)\|_w$, where $w(\cdot)$ is a nonnegative sufficiently smooth weight function and $w(\cdot)$ as well as its first $(q-1)$ derivatives vanish at $0$ and $1$. Here $\bm\theta$ is defined as
\begin{eqnarray}
\bm\theta=\arg\min_{\bm\eta\in\bm\Theta}\left\|F\left(\cdot,f(\cdot,\bm\beta),f^{(1)}(\cdot,\bm\beta),\ldots,f^{(q)}(\cdot,\bm\beta),\bm\eta\right)\right\|_w,\nonumber
\end{eqnarray}
where $f(\cdot,\bm\beta)=\bm\beta^T\bm N(\cdot)$ and $f^{(r)}(t,\bm\beta)=\frac{d^r}{dt^r}f(t,\bm\beta)$ for $r=1,\ldots,q$. We also assume that for all $\epsilon>0$
\begin{eqnarray}
\inf_{\bm\eta:\norm{\bm\eta-\bm\theta_0}\geq\epsilon}\norm{F\left(\cdot,f_0(\cdot),f_0^{(1)}(\cdot),\ldots,f_0^{(q)}(\cdot),\bm\eta\right)}_w
>\norm{F\left(\cdot,f_0(\cdot),f_0^{(1)}(\cdot),\ldots,f_0^{(q)}(\cdot),\bm\theta_0\right)}_w\label{assmp_ts_hode}
\end{eqnarray}
which implies that $\bm\theta_0$ is a well separated point of minima of $\norm{F\left(\cdot,f_0(\cdot),f_0^{(1)}(\cdot),\ldots,f_0^{(q)}(\cdot),\bm\eta\right)}_w$.
\section{Main results}\label{result_hode}
Now we state the theoretical results corresponding to each of the three approaches.
\subsection{RKSB}
\begin{theorem}\label{bvm_rksb_hode}
Let the posterior probability measure related to RKSB be denoted by $\Pi_n$. Then posterior of $\bm\gamma$ contracts at $\bm\gamma_0$ at the rate $n^{-1/2}$ and
\begin{eqnarray}
\norm{\Pi_n\left(\sqrt{n}(\bm\gamma-\bm\gamma_0)\in\cdot|\bm X,\bm Y\right)-\bm N(\bm\mu_n,\bm\Sigma)}_{TV}\stackrel{P_0}\rightarrow0\nonumber
\end{eqnarray}
where $\bm\Sigma=\left({\begin{array}{cc}{\sigma^{2}_*}{\bm V^{-1}_{\bm\theta_0}}&\bm 0\\\bm 0&{2\sigma^{4}_*}\end{array}}\right)$ with
\begin{eqnarray}
\bm V_{\bm\theta_0}=\int_0^1\left(\dot{f}^T_{\bm\theta_0}(t)\dot{f}_{\bm\theta_0}(t)-\frac{\partial}{\partial\bm\theta}\left.\left(\dot{ f}^T_{\bm\theta}(t)\left(f_0(t)-f_{\bm\theta_0}(t)\right)\right)\right|_{\bm\theta=\bm\theta_0}\right)g(t)dt\nonumber
\end{eqnarray}
and $\bm\mu_n=\bm\Sigma\mathbb{G}_n\dot\ell_{\bm\gamma_0,n}$.
\end{theorem}
Since $\bm\theta$ is a sub-vector of $\bm\gamma$, we get the Bernstein-von Mises Theorem for the posterior distribution of $\sqrt{n}(\bm\theta-\bm\theta_0)$, the mean and dispersion matrix of the limiting Gaussian distribution being the corresponding sub-vector and sub-matrix of $\bm\mu_n$ and $\bm\Sigma$ respectively. We also get the following important corollary.
\begin{corollary}\label{eff_rksb_hode}
When the regression model \eqref{prop_hode} is correctly specified and also the true distribution of error is Gaussian, the Bayes estimator based on $\Pi_n$ is asymptotically efficient.
\begin{remark}
\rm{We can get similar results when the covariates are deterministic under the criteria that $$\sup_{t\in[0,1]}\left|Q_n(t)-Q(t)\right|=o(n^{-1/2}),$$ where $Q_n(\cdot)$ is the empirical distribution function of the covariate sample and $Q(\cdot)$ is a distribution function with positive pdf on [0,1].}
\end{remark}
\end{corollary}
\subsection{RKTB}
In RKTB we assume that the matrix
\begin{eqnarray}
\bm J({\bm\theta_0})=-\int_0^1\ddot f_{\bm\theta_0}(t)(f_0(t)- f_{\bm\theta_0}(t))g(t)dt
+\int_0^1\left(\dot f_{\bm\theta_0}(t)\right)^T\left(\dot f_{\bm\theta_0}(t)\right)g(t)dt\nonumber
\end{eqnarray}
is nonsingular. Note that in the well-specified case the first term vanishes and hence $\bm J(\bm\theta_0)$ equals the second term which is positive definite. Let us denote $\bm C(t)=\left(\bm J({\bm\theta_0})\right)^{-1}\left(\dot f_{\bm\theta_0}(t)\right)^T$ and $\bm G^T_n=\int_0^1\bm C(t)\bm N^T(t)g(t)dt$. Also, we denote the posterior probability measure of RKTB by $\Pi_n^*$. Now we have the following result.

\begin{theorem}\label{bvm_rktb_hode}
Let
\begin{eqnarray}
\bm\mu_n^*&=&\sqrt{n}\bm G_{n}^T\left(\bm X^T_n\bm X_n\right)^{-1}\bm X^T_n\bm Y-\sqrt{n}\left(\bm J({\bm\theta_0})\right)^{-1}\int_0^1\left(\dot f_{\bm\theta_0}(t)\right)^Tf_0(t)g(t),\nonumber\\
\bm\Sigma_n^*&=&n\bm G_{n}^T\left(\bm X^T_n\bm X_n\right)^{-1}\bm G_{n},\nonumber\\
\bm B&=&\left(\!\!\left(\left<C_{k}(\cdot),C_{k'}(\cdot)\right>_g\right)\!\!\right)_{k,k'=1,\ldots,p}.\nonumber
\end{eqnarray}
If $\bm B$ is non-singular, then for $m\geq3$ and $n^{1/(2m)}\ll k_n\ll n^{1/4}$,
\begin{eqnarray}
\norm{\Pi^*_n\left(\sqrt{n}(\bm\theta-\bm\theta_0)\in\cdot|\bm X,\bm Y\right)-N\left(\bm\mu_n^*,\sigma^2_0\bm\Sigma_n^*\right)}_{TV}=o_{P_0}(1).\label{thm2_hode_rktb}
\end{eqnarray}
\end{theorem}
\begin{remark}
\rm{Following the steps of the proof of Lemma 10 of \citet{bhaumik2014efficient} it can be proved that both $\bm\mu^*_n$ and $\bm\Sigma^*_n$ are stochastically bounded. Hence, with high true probability the posterior distribution of $\bm\theta$ contracts at $\bm\theta_0$ at $n^{-1/2}$ rate}.
\end{remark}
We also get the following important corollary.
\begin{corollary}\label{eff_rktb_hode}
When the regression model \eqref{prop_hode} is correctly specified and the true distribution of error is Gaussian, the Bayes estimator based on $\Pi^*_n$ is asymptotically efficient.
\end{corollary}
\begin{remark}
\rm{Similar results will follow for deterministic covariates provided that $$\sup_{t\in[0,1]}\left|Q_n(t)-Q(t)\right|=o(k_n^{-1}),$$ where $Q_n(\cdot)$ is the empirical distribution function of the covariate sample and $Q(\cdot)$ is a distribution function with positive pdf on [0,1]. Note that this condition holds with probability tending to one when the covariates are random.}
\end{remark}

\subsection{Two-step Bayesian Method}
We denote $\bm h(\cdot)=(f(\cdot,\bm\beta),f^{(1)}(\cdot,\bm\beta),\ldots,f^{(q)}(\cdot,\bm\beta))^T$ and $\bm h_0(\cdot)$ to be similar to $\bm h(\cdot)$ with $f$ being replaced by $f_0$. We denote $\bm G(t,\bm h(t),\bm\theta)=\left(D_{\bm\theta}F(t,\bm h(t),\bm\theta)\right)^TF(t,\bm h(t),\bm\theta)$. Before obtaining the Bernstein-von Mises Theorem for two-step Bayesian method, we use the following lemma to get an approximate linearization of $\sqrt{n}(\bm\theta-\bm\theta_0)$.
\begin{lemma}\label{linearize_ts_hode}
Let the matrix
\begin{eqnarray}
\bm M(\bm h_0,\bm\theta_0)=\int_0^1D_{\bm\theta_0}\left(\bm G(t,\bm h_0(t),\bm\theta_0)\right)w(t)dt\nonumber
\end{eqnarray}
be nonsingular. For $m>(2q+2)$, $n^{1/2m}\ll k_n\ll n^{1/(4q+4)}$ and under assumption \eqref{assmp_ts_hode}, there exists $E_n\subseteq\bm\Theta\times C^m((0,1))$ with $\Pi(E^c_n|\bm Y)=o_{P_0}(1)$, such that
\begin{eqnarray}
\sup_{(\bm\theta,\bm h)\in E_n}\norm{\sqrt{n}(\bm\theta-\bm\theta_0)-\left(\bm M(\bm h_0,\bm\theta_0)\right)^{-1}\sqrt{n}(\bm\Gamma(f)-\bm\Gamma(f_{0}))}\rightarrow0,\label{thm1_hode_ts}
\end{eqnarray}
where $\bm\Gamma(z):=-\sum_{r=0}^q\int_0^1(-1)^{r}\frac{d^r}{dt^r}\left[D_{\bm h_0}\left(\bm G(t,\bm h_0(t),\bm\theta_0)\right)w(t)\right]_{,r}z(t)dt$  is a linear functional of $z(\cdot)$ for any function $z:[0,1]\mapsto\mathbb{R}$.
\end{lemma}
Denoting $\bm A(t)=-\left(\bm M(\bm h_0,\bm\theta_0)\right)^{-1}\sum_{r=0}^q(-1)^{r}\frac{d^r}{dt^r}\left[D_{\bm h_0}\left(\bm G(t,\bm h_0(t),\bm\theta_0)\right)w(t)\right]_{,r}$, we have
\begin{eqnarray}
\left(\bm M(\bm h_0,\bm\theta_0)\right)^{-1}\bm\Gamma(f)=\int_0^1\bm A(t)\bm\beta^T\bm N(t)d\bm t=\bm H_{n}^T\bm\beta,\label{linearize_hode}
\end{eqnarray}
where $\bm H_{n}^T=\int_0^1\bm A(t)\bm N^T(t)dt$ which is a matrix of order $p\times (k_n+m-1)$. 
Then in order to approximate the posterior distribution of $\bm\theta$, it suffices to study the asymptotic posterior distribution of the linear functional of $\bm\beta$ given by \eqref{linearize_hode}. The next theorem describes the approximate posterior distribution of $\sqrt{n}(\bm\theta-\bm\theta_0)$. We denote the posterior probability measure of two-step method by $\Pi_n^{**}$.

\begin{theorem}\label{bvm_ts_hode}
Let us denote
\begin{eqnarray}
\bm\mu^{**}_n&=&\sqrt{n}\bm H_{n}^T(\bm X^T_n\bm X_n)^{-1}\bm X^T_n\bm Y-\sqrt{n}\left(\bm M(\bm h_0,\bm\theta_0)\right)^{-1}\bm\Gamma(f_0),\nonumber\\
\bm\Sigma^{**}_n&=&n\bm H_{n}^T(\bm X^T_n\bm X_n)^{-1}\bm H_{n}\nonumber
\end{eqnarray}
and $\bm D=\left(\!\!\left(\left<A_{k}(\cdot),A_{k'}(\cdot)\right>_g\right)\!\!\right)_{k,k'=1,\ldots,p}$. If $\bm D$ is non-singular, then under the conditions of Lemma \ref{linearize_ts_hode},
\begin{eqnarray}
\norm{\Pi^{**}_n\left(\sqrt{n}(\bm\theta-\bm\theta_0)\in\cdot|\bm X,\bm Y\right)-N\left(\bm\mu^{**}_n,\sigma^2_0\bm\Sigma^{**}_n\right)}_{TV}=o_{P_0}(1).\label{thm2_hode_ts}
\end{eqnarray}
\end{theorem}
As in RKSB and RKTB, we get similar results for two-step Bayesian approach when the regressor is non-random under appropriate conditions.

\section{Simulation Study}\label{sim_hode}
We consider the van der Pol equation
\begin{eqnarray}
\frac{d^2f_{\theta}(t)}{dt^2}-\theta(1-f^2_{\theta}(t))\frac{df_{\theta}(t)}{dt}+f_{\theta}(t)=0,\,\,t\in[0,1]\nonumber
\end{eqnarray}
with the initial conditions $f_{\theta}(0)=2, f'_{\theta}(0)=0$, to study the posterior distribution of $\theta$. The above system is not analytically solvable. We consider the situation where the true regression function belongs to the solution set. For a sample of size $n$, the predictor variables $X_1,\ldots,X_n$ are drawn from the $\mathrm{Uniform}(0,1)$ distribution. Samples of sizes 100 and 500 are considered. We simulate 1000 replications for each case. Under each replication a sample of size 1000 is drawn from the posterior distribution of $\theta$ using RKSB, RKTB and Bayesian two-step methods and then 95\% equal tailed credible intervals are obtained. The simulation results are summarized in the Table \ref{table_hode}. Bayesian two-step method is abbreviated as ``TS" in the table. We calculate the coverage and the average length of the corresponding credible interval over these $1000$ replications. We also compare these three methods with the nonlinear least squares (NLS) technique based on exhaustive numerical solution of the ODE where we construct 95\% confidence interval using asymptotic normality. The estimated standard errors of the interval length and coverage are given inside the parentheses in the table. 
The true parameter vector is chosen as $\theta_0=1.$ The true distribution of error is taken $N(0,(0.1)^2)$. We put an inverse gamma prior on $\sigma^2$ with shape and scale parameters being $99$ and $1$ respectively. For RKSB the prior for $\theta$ is chosen as independent Gaussian distribution with mean $6$ and variance $16$. We take $n$ grid points to obtain the numerical solution of the ODE by RK4 for a sample of size $n$. We take $m=5$ and $m=7$ for RKTB and Bayesian two-step method respectively. Looking at the order of $k_n$ suggested by Theorem \ref{bvm_rktb_hode}, $k_n$ is chosen as $3$ and $4$ for $n=100$ and $n=500$ respectively in RKTB. In Bayesian two-step method the choices are $2$ and $3$ for $n=100$ and $n=500$ respectively. The weight function for TS is chosen as $w(t)=t^2(1-t)^2$. 
$$\rm{``Table\,\ref{table_hode}\,here"}$$
Note the similarity in the outputs corresponding to RKTB, RKSB and NLS because of asymptotic efficiency while TS intervals are much wider. However, TS is computationally much faster.

\section{Proofs}\label{proof_hode}
We use the operators $\E_0(\cdot)$ and $\Var_0(\cdot)$ to denote expectation and variance with respect to $P_0$.
\begin{proof}[Proof of Theorem \ref{bvm_rksb_hode}]
As in Lemma 1 of \citet{bhaumik2014efficient} we can argue that there exists a compact subset $U$ of $(0,\infty)$ such that $\Pi_n(\sigma^2\in U|\bm X,\bm Y)\stackrel{P_0}\rightarrow1$. Let $\Pi_{U,n}(\cdot|\bm X,\bm Y)$ be the posterior distribution conditioned on $\sigma^2\in U$. By Theorem 2.1 of \citet{kleijn2012bernstein} if we can ensure that there exist stochastically bounded random variables $\bm\mu_n$ and a positive definite matrix $\bm\Sigma$ such that for every compact set $K\subset\mathbb{R}^{p+1}$,
\begin{eqnarray}
\sup_{\bm h\in K}\left|\log\frac{p^{(n)}_{\bm\gamma_0+\bm h/\sqrt{n},n}}{p^{(n)}_{\bm\gamma_0,n}}(\bm X,\bm Y)-\bm h^T\bm\Sigma^{-1}\bm\mu_n+\frac{1}{2}\bm h^T\bm\Sigma^{-1}\bm h\right|\rightarrow0,\label{cond1_hode}
\end{eqnarray}
in (outer) $P_0^{(n)}$-probability and that for every sequence of constants $M_n\rightarrow\infty$, we have
\begin{eqnarray}
P^{(n)}_0\Pi_{U,n}\left(\sqrt{n}\norm{\bm\gamma-\bm\gamma_0}>M_n|\bm X,\bm Y\right)\rightarrow0,\label{cond2_hode}
\end{eqnarray}
then $\norm{\Pi_{U,n}\left(\sqrt{n}(\bm\gamma-\bm\gamma_0)\in\cdot|\bm X,\bm Y\right)-\bm N(\bm\mu_n,\bm\Sigma)}_{TV}\stackrel{P_0}\rightarrow0$. To show that the conditions \eqref{cond1_hode} and \eqref{cond2_hode} hold, we prove results similar to Lemmas 2 to 5 of \citet{bhaumik2014efficient}. Following the steps of Lemma 2 of \citet{bhaumik2014efficient} we get $\bm\Sigma=\left({\begin{array}{cc}{\sigma^{2}_*}{\bm V^{-1}_{\theta_0}}&\bm 0\\\bm 0&{2\sigma^{4}_*}\end{array}}\right)$ with $$\bm V_{\bm\theta_0}=\int_0^1\left(\dot{f}^T_{\bm\theta_0}(t)\dot{f}_{\bm\theta_0}(t)-\frac{\partial}{\partial\bm\theta}\left.\left(\dot{ f}^T_{\bm\theta}(t)\left(f_0(t)-f_{\bm\theta_0}(t)\right)\right)\right|_{\bm\theta=\bm\theta_0}\right)g(t)dt$$ and $\bm\mu_n=\bm\Sigma\mathbb{G}_n\dot\ell_{\bm\gamma_0,n}$. Finally we get
\begin{eqnarray}
\norm{\Pi_n\left(\sqrt{n}(\bm\gamma-\bm\gamma_0)\in\cdot|\bm X,\bm Y\right)-\bm N(\bm\mu_n,\bm\Sigma)}_{TV}\stackrel{P_0}\rightarrow0\nonumber
\end{eqnarray}
since $\norm{\Pi_{n}-\Pi_{U,n}}_{TV}=o_{P_0}(1)$ and the result follows.
\end{proof}
\begin{proof}[Proof of Corollary \ref{eff_rksb_hode}]
The log-likelihood of the correctly specified model with Gaussian error is given by
\begin{eqnarray}
\ell_{\bm\gamma_0}(X,Y)&=&-\log\sigma_0-\frac{1}{2\sigma^2_0}|Y-f_{\bm\theta_0}(X)|^2+\log g(X)\nonumber.
\end{eqnarray}
Thus $\frac{\partial}{\partial\bm\theta_0}\ell_{\bm\gamma_0}(X,Y)=\sigma^{-2}_0\left(\dot{f}_{\bm\theta_0}(X)\right)^T\left(Y-f_{\bm\theta_0}(X)\right)$ and $\frac{\partial}{\partial\sigma^2_0}\ell_{\bm\gamma_0}(X,Y)=-\frac{1}{2\sigma^2_0}+\frac{1}{2\sigma^4_0}|Y-f_{\bm\theta_0}(X)|^2$. Hence, the Fisher information is given by
\begin{eqnarray}
\bm I(\bm\gamma_0)=\left({\begin{array}{cc}{{\sigma^{-2}_0}\int_0^1\dot{f}^T_{\bm\theta_0}(t)\dot{f}_{\bm\theta_0}(t)g(t)dt}&\bm 0\\\bm 0&{\sigma^{-4}_0}/2\end{array}}\right).\nonumber
\end{eqnarray}
Thus $\bm\Sigma=\left(\bm I(\bm\gamma_0)\right)^{-1}$ if the regression function is correctly specified and the true error distribution is $N(0,\sigma^2_0)$.
\end{proof}
\begin{proof}[Proof of Theorem \ref{bvm_rktb_hode}]
For $f(\cdot,\bm\beta)=\bm\beta^T\bm N(\cdot)$ we have $\int_0^1\bm C(t)\bm\beta^T\bm N(t)g(t)dt=\bm G_n^T\bm\beta$, where $$\bm G_{n}^T=\int_0^1\bm C(t)\bm N^T(t)g(t)dt$$ which is a matrix of order $p\times (k_n+m-1)$. 
We can derive the posterior consistency of $\sigma^2$ similar to Lemma 11 of \citet{bhaumik2014efficient}. Following result similar to Lemma 9 of \citet{bhaumik2014efficient}, it can be shown that on a set with high posterior probability
\begin{eqnarray}
\norm{\sqrt{n}(\bm\theta-\bm\theta_0)-\left(\sqrt{n}\bm G_{n}^T\bm\beta-\sqrt{n}\left(\bm{J}({\bm\theta_0})\right)^{-1}\int_0^1\left(\dot f_{\bm\theta_0}(t)\right)^Tf_0(t)g(t)\right)}\rightarrow0\nonumber
\end{eqnarray}
as $n\rightarrow\infty$. Then it suffices to show that for any neighborhood $\mathscr{N}$ of $\sigma^2_0$,
\begin{eqnarray}
\sup_{\sigma^2\in\mathscr{N}}\left\|\Pi^*_n\left(\sqrt{n}\bm G_{n}^T\bm\beta-\sqrt{n}\left(\bm{J}({\bm\theta_0})\right)^{-1}\int_0^1\left(\dot f_{\bm\theta_0}(t)\right)^Tf_0(t)g(t)\in\cdot|\bm X,\bm Y,\sigma^2\right)-{N}(\bm\mu_n^*,\sigma^2\bm\Sigma_n^*)\right\|_{TV}\nonumber
\end{eqnarray}
is $o_{P_0}(1)$. The rest of the proof follows from that of Theorem 4.2 of \citet{bhaumik2014efficient}.
\end{proof}

\begin{proof}[Proof of Corollary \ref{eff_rktb_hode}]
The log-likelihood of the correctly specified model is given by
\begin{eqnarray}
\ell_{\bm\theta_0}(X,Y)&=&-\log\sigma_0-\frac{1}{2\sigma^2_0}|Y-f_{\bm\theta_0}(X)|^2+\log g(X)\nonumber.
\end{eqnarray}
Thus $\dot\ell_{\bm\theta_0}(X,Y)=-\sigma^{-2}_0\left(\dot{f}_{\bm\theta_0}(X)\right)^T\left(Y-f_{\bm\theta_0}(X)\right)$ and the Fisher information is given by
$\bm I(\bm\theta_0)=\sigma^{-2}_0\int_0^1\left(\dot{f}_{\bm\theta_0}(X)\right)^T\dot{f}_{\bm\theta_0}(t)g(t)dt$.
Following the proof of Lemma 10 of \citet{bhaumik2014efficient} we get $$\sigma^2_0\bm\Sigma_n^*\stackrel{P_0}\rightarrow\sigma^2_0\left(\bm{J}({\bm\theta_0})\right)^{-1}\int_0^1\left(\dot f_{\bm\theta_0}(t)\right)^T\dot f_{\bm\theta_0}(t)g(t)dt\,\left(\left(\bm{J}({\bm\theta_0})\right)^{-1}\right)^T.$$ This limit is equal to $\left(\bm I({\bm\theta_0})\right)^{-1}$ under the correct specification of the regression function as well as the error model.
\end{proof}

\begin{proof}[Proof of Lemma \ref{linearize_ts_hode}]
By the definitions of $\bm\theta$ and $\bm\theta_0$ we have
\begin{eqnarray}
\int_0^1\bm G(t,\bm h(t),\bm\theta)w(t)dt=\bm 0,\,
\int_0^1\bm G(t,\bm h_0(t),\bm\theta_0)w(t)dt=\bm 0.\nonumber
\end{eqnarray}
Subtracting the second equation from the first and applying the Mean-value Theorem, we get
\begin{eqnarray}
\,\,\int_0^1D_{\bm\theta_0}\left(\bm G(t,\bm h_0(t),\bm\theta_0)\right)w(t)dt(\bm\theta-\bm\theta_0)+\int_0^1D_{\bm h_0}\left(\bm G(t,\bm h_0(t),\bm\theta_0)\right)w(t)(\bm h(t)-\bm h_{0}(t))dt\nonumber\\
+O\left(\sup_{t\in[0,1]}\norm{\bm h(t)-\bm h_0(t)}^2\right)+O(\norm{\bm\theta-\bm\theta_0}^2)=\bm 0.\nonumber
\end{eqnarray}
Now we will show that $\int_0^1D_{\bm h_0}\left(\bm G(t,\bm h_0(t),\bm\theta_0)\right)w(t)(\bm h(t)-\bm h_{0}(t))dt$ is a linear functional of $f-f_0$. Note that  $\int_0^1D_{\bm h_0}\left(\bm G(t,\bm h_0(t),\bm\theta_0)\right)w(t)(\bm h(t)-\bm h_{0}(t))dt$ can be written as $$\sum_{r=0}^q\int_0^1\left[D_{\bm h_0}\left(\bm G(t,\bm h_0(t),\bm\theta_0)\right)w(t)\right]_{,r}\left(f^{(r)}(t,\bm\beta)-f^{(r)}_0(t)\right)dt.$$ We shall show that every term of this sum is a linear functional of $f-f_0$. We observe that for $r=0,\ldots,q$
\begin{eqnarray}
\lefteqn{\int_0^1\left[D_{\bm h_0}\left(\bm G(t,\bm h_0(t),\bm\theta_0)\right)w(t)\right]_{,r}\left(f^{(r)}(t,\bm\beta)-f^{(r)}_0(t)\right)dt}\nonumber\\
&=&(-1)^r\int_0^1\frac{d^r}{dt^r}\left[D_{\bm h_0}\left(\bm G(t,\bm h_0(t),\bm\theta_0)\right)w(t)\right]_{,r}(f(t,\bm\beta)-f_{0}(t))dt\nonumber
\end{eqnarray}
using integration by parts and the fact that the function $w(\cdot)$ and its first $(q-1)$ derivatives vanish at $0$ and $1$. Proceeding this way we get
\begin{eqnarray}
\bm M(\bm h_0,\bm\theta_0)(\bm\theta-\bm\theta_0)-\bm\Gamma(f-f_0)+O\left(\sup_{t\in[0,1]}\norm{\bm h(t)-\bm h_0(t)}^2\right)+O(\norm{\bm\theta-\bm\theta_0}^2)=\bm 0.\nonumber
\end{eqnarray}
Let $E_n=\{(\bm h,\bm\theta):\sup_{t\in[0,1]}\norm{\bm h(t)-\bm h_{0}(t)}\leq\epsilon_n, \norm{\bm\theta-\bm\theta_0}\leq\epsilon_n\}$, where $\epsilon_n\rightarrow0$. Using the steps of the proofs of Lemmas 2 and 3 of \citet{bhaumik2014bayesiant}, we can prove the posterior consistency of $\bm\theta$. Hence, there exists a sequence $\{\epsilon_n\}$ so that $\Pi(E^c_n|\bm Y)=o_{P_0}(1)$. Hence, on $E_n$
\begin{eqnarray}
\sqrt{n}(\bm\theta-\bm\theta_0)=\left((\bm M(\bm h_0,\bm\theta_0))^{-1}+o(1)\right)\sqrt{n}\bm\Gamma(f-f_0)+\sqrt{n}\sup_{t\in[0,1]}\norm{\bm h(t)-\bm h_{0}(t)}^2O(1).\nonumber
\end{eqnarray}
By result similar to Lemma 4 of \citet{bhaumik2014bayesiant}, $\sqrt{n}\,\bm\Gamma(f-f_0)$ assigns most of its mass inside a large compact set. Now using result similar to Lemma 2 of \citet{bhaumik2014bayesiant}, we can assert that on $E_n$, the second term on the display is $o(1)$ and the conclusion follows.
\end{proof}
\begin{proof}[Proof of Theorem \ref{bvm_ts_hode}]
By Lemma \ref{linearize_ts_hode} and \eqref{linearize_hode}, it suffices to show that for any $\sigma^2$ in a neighborhood of $\sigma^2_0$,
\begin{eqnarray}
\norm{\Pi\left(\sqrt{n}\bm H_{n}^T\bm\beta-\sqrt{n}\left(\bm M(\bm h_0,\bm\theta_0)\right)^{-1}\bm\Gamma(f_0)\in\cdot|\bm X,\bm Y\right)-{N}(\bm\mu_n^{**},\sigma^2\bm\Sigma_n^{**})}_{TV}
=o_{P_0}(1).\nonumber\\\label{thm4_1}
\end{eqnarray}
Note that the posterior distribution of $\bm H^T_{n}\bm\beta$ is a normal distribution with mean vector $$\bm H^T_n{\left({\bm X^T_n\bm X_n}+n^{-2}k_n\bm I_{k_n+m-1}\right)}^{-1}{\bm X^T_n\bm Y}$$ and dispersion matrix $$\sigma^2\bm H^T_n{\left({\bm X^T_n\bm X_n}+n^{-2}k_n\bm I_{k_n+m-1}\right)}^{-1}\bm H_{n}$$ respectively. We calculate the Kullback-Leibler divergence between two Gaussian distributions and show that it converges in $P_0$-probability to zero to prove the assertion. The rest of the proof is similar to that of Theorem 3 of \citet{bhaumik2014bayesiant}.
\end{proof}

\begin{table}[h]
\centering
\caption{\textit{Arguments of $\bm H$}}\label{arg_table}
\begin{tabular}{lllll}
$\rho$&1&2&3&4\\
\hline
$\bm U^1$&$\bm\psi_{\bm\theta,r_n}(a_k)$&$\bm\psi_{\bm\theta,r_n}(a_k)$&$ \bm\psi_{\bm\theta,r_n}(a_k)+\frac{1}{2r_n}\bm\psi_{\bm\theta,r_n}^{(1)}(a_k)$&$ \bm\psi_{\bm\theta,r_n}(a_k)+\frac{1}{r_n}\bm\psi_{\bm\theta,r_n}^{(1)}(a_k)$\\
&&$+\frac{1}{2r_n}\bm\psi_{\bm\theta,r_n}^{(1)}(a_k)$&$+\frac{1}{4r^2_n}\bm\psi_{\bm\theta,r_n}^{(2)}(a_k)$&$+\frac{1}{2r^2_n} \bm\psi_{\bm\theta,r_n}^{(2)}(a_k)$\\
&&&&$+\frac{1}{4r^3_n}\bm\psi_{\bm\theta,r_n}^{(3)}(a_k)$\\
&&&&\\
$\bm U^2$&$\bm\psi_{\bm\theta,r_n}^{(1)}(a_k)$&$\bm\psi_{\bm\theta,r_n}^{(1)}(a_k)$&$\bm\psi_{\bm\theta,r_n}^{(1)}(a_k)+\frac{1}{2r_n}\bm\psi_{\bm\theta,r_n}^{(2)}(a_k)$&$\bm\psi_{\bm\theta,r_n}^{(1)}(a_k)+\frac{1}{r_n}\bm\psi_{\bm\theta,r_n}^{(2)}(a_k)$\\
&&$+\frac{1}{2r_n}\bm\psi_{\bm\theta,r_n}^{(2)}(a_k)$&$+\frac{1}{4r^2_n} \bm\psi_{\bm\theta,r_n}^{(3)}(a_k)$&$+\frac{1}{2r^2_n}\bm\psi_{\bm\theta,r_n}^{(3)}(a_k)$\\
&&&&$+\frac{1}{4r^3_n}\bm\psi_{\bm\theta,r_n}^{(4)}(a_k)$\\
\vdots&\vdots&\vdots&\vdots&\vdots\\
$\bm U^{q-3}$&$\bm\psi_{\bm\theta,r_n}^{(q-4)}(a_k)$&$\bm\psi_{\bm\theta,r_n}^{(q-4)}(a_k)$&$\bm\psi_{\bm\theta,r_n}^{(q-4)}(a_k)+\frac{1}{2r_n} \bm\psi_{\bm\theta,r_n}^{(q-3)}(a_k)$&$\bm\psi_{\bm\theta,r_n}^{(q-4)}(a_k)+\frac{1}{r_n}\bm\psi_{\bm\theta,r_n}^{(q-3)}(a_k)$\\
&&$+\frac{1}{2r_n}\bm\psi_{\bm\theta,r_n}^{(q-3)}(a_k)$&$+\frac{1}{4r^2_n}\bm\psi_{\bm\theta,r_n}^{(q-2)}(a_k)$&$+\frac{1}{2r^2_n}\bm\psi_{\bm\theta,r_n}^{(q-2)}(a_k)$\\
&&&&$+\frac{1}{4r^3_n}\bm\psi_{\bm\theta,r_n}^{(q-1)}(a_k)$\\
&&&&\\
$\bm U^{q-2}$&$\bm\psi_{\bm\theta,r_n}^{(q-3)}(a_k)$&$\bm\psi_{\bm\theta,r_n}^{(q-3)}(a_k)$&$\bm\psi_{\bm\theta,r_n}^{(q-3)}(a_k)+\frac{1}{2r_n} \bm\psi_{\bm\theta,r_n}^{(q-2)}(a_k)$&$\bm\psi_{\bm\theta,r_n}^{(q-3)}(a_k)+\frac{1}{r_n}\bm\psi_{\bm\theta,r_n}^{(q-2)}(a_k)$\\
&&$+\frac{1}{2r_n}\bm\psi_{\bm\theta,r_n}^{(q-2)}(a_k)$&$+\frac{1}{4r^2_n}\bm\psi_{\bm\theta,r_n}^{(q-1)}(a_k)$&$+\frac{1}{2r^2_n}\bm\psi_{\bm\theta,r_n}^{(q-1)}(a_k)+\frac{1}{4r^3_n}\bm k_1$\\
&&&&\\
$\bm U^{q-1}$&$\bm\psi_{\bm\theta,r_n}^{(q-2)}(a_k)$&$\bm\psi_{\bm\theta,r_n}^{(q-2)}(a_k)$&$\bm\psi_{\bm\theta,r_n}^{(q-2)}(a_k)+\frac{1}{2r_n} \bm\psi_{\bm\theta,r_n}^{(q-1)}(a_k)$&$\bm\psi_{\bm\theta,r_n}^{(q-2)}(a_k)+\frac{1}{r_n}\bm\psi_{\bm\theta,r_n}^{(q-1)}(a_k)$\\
&&$+\frac{1}{2r_n}\bm\psi_{\bm\theta,r_n}^{(q-1)}(a_k)$&$+\frac{1}{4r^2_n}\bm k_1$&$+\frac{1}{2r^2_n}\bm k_2$\\
&&&&\\
$\bm U^{q}$&$\bm\psi_{\bm\theta,r_n}^{(q-1)}(a_k)$&$\bm\psi_{\bm\theta,r_n}^{(q-1)}(a_k)$&$\bm\psi_{\bm\theta,r_n}^{(q-1)}(a_k)+\frac{1}{2r_n}\bm k_2$&$\bm\psi_{\bm\theta,r_n}^{(q-1)}(a_k)+\frac{1}{r_n}\bm k_3$\\
&&$+\frac{1}{2r_n}\bm k_1$&&\\
\hline
\end{tabular}
\end{table}

\begin{table}[ht]
\centering
\caption{\textit{Coverages and average lengths of the Bayesian credible intervals and confidence intervals}}\label{table_hode}
\begin{tabular}{|c|c|cc|cc|cc|cc|}
\hline
$n$&&\multicolumn{2}{|c|}{RKTB}&\multicolumn{2}{|c|}{RKSB}&\multicolumn{2}{|c|}{TS}&\multicolumn{2}{|c|}{NLS}\\
&&coverage&length&coverage&length&coverage&length&coverage&length\\
&&(se)&(se)&(se)&(se)&(se)&(se)&(se)&(se)\\
\hline
100&$\theta$&95.6&0.34&94.7&0.33&94.9&1.95&95.2&0.32\\
&&(0.02)&(0.06)&(0.02)&(0.04)&(0.02)&(0.71)&(0.02)&(0.03)\\
\cline{2-10}
500&$\theta$&95.7&0.14&95.4&0.14&96.7&0.70&95.1&0.14\\
&&(0.01)&(0.01)&(0.01)&(0.01)&(0.01)&(0.13)&(0.01)&(0.01)\\
\hline
\end{tabular}
\end{table}
\bibliographystyle{apalike}
\bibliography{ref_hode}
\vspace{.5in}
Author address: Prithwish Bhaumik\\
Department of Statistics\\
North Carolina State University\\
SAS Hall, 2311 Stinson Drive\\
Raleigh, NC 27695-8203\\
USA.\\
Email: prithwish1987@gmail.com

\end{document}